\documentclass[final]{scrartcl}
\usepackage{mathtools}  
\usepackage{amsmath}
\usepackage{amssymb}
\usepackage{amsthm}
\usepackage{amsxtra}
\usepackage{xcolor}

\usepackage[hyperfootnotes=false,english]{hyperref}
\usepackage[english]{babel}
\usepackage[utf8]{inputenc}

\usepackage{url}

\usepackage[numbers]{natbib}

\newcommand*{\field}[1]{\mathbb{#1}}
\newcommand*{\Q}{\field{Q}}
\newcommand*{\R}{\field{R}}
\newcommand*{\C}{\field{C}}
\newcommand*{\Z}{\field{Z}}
\newcommand*{\N}{\field{N}}

\newcommand*{\HH}{\mathbb{H}}

\newcommand*{\group}[1]{\mathrm{#1}}
\newcommand*{\SL}{\group{SL}}

\newcommand*{\nP}{r}

\DeclarePairedDelimiter{\abs}{\lvert}{\rvert}

\newtheorem{Theorem}{Theorem}
\newtheorem{BThm}{Theorem} 

\newtheorem{Lemma}{Lemma}
\newtheorem{Corollary}{Corollary}

\newtheorem*{rmk*}{Remark}
\KOMAoptions{abstract=on}

\title{On Products of Fourier Coefficients of Cusp Forms}
\author{Eric Hofmann, Winfried Kohnen} 
 \date{}
\begin{document}
\maketitle

\begin{abstract}The purpose of this paper is to study products of Fourier coefficients 
of an elliptic cusp form, $a(n)a(n+r)$ $(n \geq 1)$ for a fixed positive integer $r$,
concerning both non-vanishing and non-negativity.\footnote{ 2010 \textit{Mathematics Subject 
Classification:} 11F12, 11F30 }
\end{abstract}
\section{Introduction}
Let $f$ be an elliptic cusp form of positive integral weight $k$ with real Fourier coefficients 
$a(n)$ $(n\geq 1)$. Let $r$ be a fixed positive integer. Then, the purpose of this paper is to study 
the products $a(n) a(n+r)$ $(n\geq1)$, regarding both non-vanishing and non-negativity. For a 
precise statement of our results see section \ref{sec:results} below. 

We remark that these products have been previously investigated under different aspects, 
namely first by Selberg \cite{Selb} and by Good \cite{Good} who studied growth properties of the 
sums $\sum_{n\leq x} a(n)a(n+r)$ where $x \rightarrow \infty$, and more recently by Hoffstein and 
Hulse \cite{Hofs} in connection with shifted Dirichlet convolutions.

\section{Statement of results}\label{sec:results}

Let $\HH$ denote the upper half-plane of complex numbers $z = x + iy$, with $y>0$. The modular 
group $\Gamma_1= \SL_2(\Z)$ acts on $\HH$ by fractional linear transformations, as usual.

In the following, we denote by $\Gamma$ a discrete subgroup of $\Gamma_1$, which satisfies the 
following conditions \citep[cf.][p.\ 98]{Good}:
\renewcommand{\labelenumi}{(\roman{enumi})}
\begin{enumerate}\label{conditions}
 \item $\Gamma$ is a finitely generated Fuchsian group of the first kind.
 \item The negative identity matrix is contained in $\Gamma$.
 \item $\Gamma$ contains $M = \left(\begin{smallmatrix} 
                                     1 & b \\ 0 & 1
                                    \end{smallmatrix}\right)$
  exactly if $b$ is an integer.
\end{enumerate}
Condition (iii) means that $\Gamma$ has a cusp at $i\infty$, the stabilizer of which is 
generated by the two matrices $\pm \left( \begin{smallmatrix} 1 & 1 \\ 0 & 1 \end{smallmatrix}  
 \right)$.

We shall prove the following theorem
\begin{Theorem}\label{thm:main_thm}
Let $f$ be a cusp form of integer weight $k> 2$ on $\Gamma$ with real Fourier coefficients $a(n)$ 
$(n\geq 1)$. Let $r \in \N$ and assume that $\left( a(n) a(n+r) \right)_{n\geq 1}$ is not 
identically zero. Then, in fact, infinitely many terms of the sequence  $\left( a(n) a(n+r) 
\right)_{n\geq 1}$ are non-zero.
\end{Theorem}

The following corollary is an immediate consequence of the theorem. Recall only that for $f$ 
a normalized Hecke eigenform, the Fourier coefficients are real and $a(1)=1$.
\begin{Corollary}\label{cor:levelN}
Let $f$ be a cusp form of integer weight $k>2$ and level $N$  that is a normalized Hecke eigenform
with Fourier coefficients $a(n)$ $(n\geq 1)$. Let $r \in \N$ and suppose that $a(r + 1) \neq 0$. 
Then, the sequence $\left( a(n) a(n+r) \right)_{n\geq 1}$ has infinitely many non-vanishing 
terms. 
\end{Corollary}
From Corollary \ref{cor:levelN}, we obtain
\begin{Corollary}\label{cor:levelOne}
Let $f$ be a normalized Hecke eigenform of integer weight $k$ on $\Gamma_1$ with Fourier coefficients 
$a(n)$  $(n\geq 1)$. Then, there are infinitely many $n$ such that $a(n)$ and $a(n+1)$ are 
both non-zero.
\end{Corollary}
\begin{proof}
By results from \cite{Serre75} and \cite{Koh99}, the second Fourier coefficient $a(2)$ is 
non-zero modulo any prime $\mathfrak{p}$ lying above $5$ (in an appropriate finite extension of 
$\Q$), hence is non-zero. The assertion follows.
\end{proof}

The proof of Theorem \ref{thm:main_thm} which will begin in section \ref{sec:proofThm1} makes use of 
a Dirichlet series associated to the sequence $\left( a(n) a(n+r) \right)_{n\geq 1}$ and of its 
analytic properties. This series was introduced by Selberg \cite{Selb} and later studied by Good  \cite{Good}, see above.

Further, assuming $\Gamma$ is a congruence subgroup of level $N$, we use a result of Knopp, 
Kohnen and Pribitkin \cite{KKP03}, Theorem 1, on sign changes of Fourier coefficients of cusp forms 
to show the following:
\begin{Theorem}\label{thm:appl_thm}
Let $f$ be a cusp form of integer weight $k$ on $\Gamma$, with real Fourier 
coefficients $a(n)$ $(n\geq 1)$. Let $r \in \N$. Then the sequence
$\left( a(n) a(n+r) \right)_{n\geq 1}$ has infinitely many non-negative terms. 
\end{Theorem}
\begin{rmk*}
 In the same way, one can prove that the sequence $(a(n)a(n+r))_{n\geq 1}$ has infinitely many non-positive terms.
\end{rmk*}
The proof of this theorem is carried out in section \ref{sec:proofThm2}. We note that, 
in fact a somewhat more general statement holds, as the requirements of \cite{KKP03}, which need to 
be considered in addition to our conditions (i)--(iii), are minimal (see the remark on p.\ \pageref{rmk:more_general} below). 

Ideally, one might hope to prove a sign change result for the sequence $a(n)a(n+r)$ $(n\geq 1)$
by combining Theorems \ref{thm:main_thm} and 
\ref{thm:appl_thm} appropriately.
However we have not been able to do this.

\section{Proof of Theorem \ref{thm:main_thm}} \label{sec:proofThm1}
We shall actually prove a more general statement than the assertion of Theorem \ref{thm:main_thm}. 

We assume that the Fourier expansion of $f$ is given by $f(z) = \sum_{n\geq 1} a(n) 
e^{2\pi i n z}$ with $a(n) \in \R$ $ (\forall n\geq 1)$. 
Let $r \in \N$ be fixed. For $n\geq  1$, set  
\[ 
c_n \vcentcolon= a(n)a(n+r).
\] 
To the sequence $(c_n)_{n\in\N}$, we attach a  Dirichlet series \citep[see][]{Good, Selb} by 
setting
\begin{equation} \label{eq:def_Dseries}
 D(s, r) \vcentcolon = \sum_{n\geq 1} \frac{c_n}{\left( n + \frac{r}{2}\right) ^s} \qquad \left( 
\sigma\vcentcolon= Re(s) > k \right). 
\end{equation}

The theorem we shall prove is the following:
\begin{BThm}\label{thm:main_bar}
Assume that $c_n \geq 0$  for almost all $n\geq 1$ and that the sum 
\[ 
D(s, r)\vcentcolon= \sum_{n \geq 1} c_n ( n + \frac{r}{2})^{-s} \qquad (\sigma > k)
\] 
in fact converges for all $s \in \C$. Then $c_n = 0$ for all $n\geq 1$. 
\end{BThm}
We postpone the proof of this Theorem, and give the proof of Theorem \ref{thm:main_thm} 
first.

\begin{proof}[Proof of Theorem \ref{thm:main_thm}]
It suffices to assume that all but a finite number of 
the coefficients $c_n$ are zero and to derive a contradiction. 

Thus, let $c_{n_1}, \dotsc, c_{n_p}$, with $n_1<\dotsc < n_p$ be those which are non-zero. 
Then, the Dirichlet series $D(s,r)$ from \eqref{eq:def_Dseries} becomes a Dirichlet polynomial
\[
D(s,r) = \sum_{i=1}^p \frac{c_{n_i}}{(n_i + \frac{r}{2})^s},
\]
and hence converges in the entire $s$-plane. Also, by hypothesis almost all of its coefficients 
are $\geq 0$. Now, by applying Theorem \ref{thm:main_bar}, we find that $c_n = 0$ for all $n\geq 1$. 
This is a contradiction, and the proof is complete.
\end{proof}

The rest of this section is dedicated to the proof of Theorem \ref{thm:main_bar}. First, we
fix some notation: 
In the following, denote by $z = x + iy$ a complex variable with real part $x$ and imaginary part $y$. 
By $\mathcal{F}$ denote a fundamental domain for the action of $\Gamma$ on $\HH$, and by 
$L^2(\Gamma\backslash \HH)$ the Hilbert space of complex-valued $\Gamma$-invariant functions on 
$\HH$ that are square integrable on $\mathcal{F}$ with respect to the invariant measure $d\nu = 
\frac{dxdy}{y^2}$. 
If $f$ and $g$ are in $L^2(\Gamma\backslash \HH)$ we denote the inner product by
\[ 
 \langle f, g \rangle = \int_{\mathcal{F}} f(z)\overline{g(z)}\, d\nu(z). 
\]
Even if $f$ and $g$ do not both belong to $L^2(\Gamma\backslash \HH)$ we continue to use the 
notation for the above integral, as long as it converges absolutely. Finally, set $F(z) 
\vcentcolon= y^k\abs{ f(z)}^2$.

\begin{proof}[Proof of Theorem \ref{thm:main_bar}]
 We loosely follow the notation used by Good in \cite{Good}.
 Fix a maximal system of $\Gamma$-inequivalent cusps, $\xi_1 
= \infty$, $\xi_2, \dotsc, \xi_\kappa$. (Note that by our assumptions on $\Gamma$, this system is 
finite and contains $\infty$.) 

For $\nP\in\N$, the Poincar\'{e} series $P_0(z,s,\nP)$
  \citep[see][p.\ 13]{Good} is defined as 
 \[
 P_0(z,s,\nP) = \frac{(\pi \nP)^{s-1/2}} {\Gamma(s+\frac12)} 
 \sum_{\gamma \in \Gamma_{\infty}\backslash \Gamma}
 \Im(\gamma z)^s e\left(\nP\Re(\gamma z)\right),
 \]
 for a complex variable $s = \sigma + it$ with $\sigma>1$. Here, $\Gamma_\infty$ stands for the 
stabilizer of $\xi_1$ in $\Gamma$, consisting of the matrices  
$M = \left(\begin{smallmatrix}  1 & b \\ 0 & 1 \end{smallmatrix}\right)$, $b\in\Z$. 

For each cusp $\xi_i$, denote by $g_i\in \Gamma_1$ an element with $g_i(\xi) = \infty$ 
and for which $\Gamma_{\xi_i}\vcentcolon= g_i^{-1} \Gamma_\infty g_i$ is the stabilizer of $\xi_i$ 
in $\Gamma$.
Then, the non-holomorphic Eisenstein series of weight zero attached to the cusp $\xi_i$ 
 (cf.\ \citep[pp.\ 104f]{Good} or e.g.\ \citep[Chapter II]{Kub}) is defined as
 \[ 
 E_i(z,s) 
 = \sum_{\gamma \in \Gamma_{\xi_i} \backslash \Gamma}  \left( \Im \left(g_i^{-1}\gamma 
z\right) 
\right)^s
 \quad(\sigma>1).
 \]
This sum converges absolutely in the half-plane $\sigma>1$. The Eisenstein series has 
meromorphic continuation to the entire $s$-plane.  Its Fourier expansion is given by 
\[
E_i(z,s) = \delta_{i,1}y^{s} + \phi_{i}(s)y^{1-s} + \sum_{m\geq 1}\phi_{i}(m; s)\,2y^{1/2} 
K_{s-1/2} (2\pi\abs{m} y) e(mx),
\]
with coefficient functions 
\[
\begin{gathered}
\phi_{i}(s) = \frac{\sqrt{\pi}\cdot \Gamma(s+\frac12)}{\Gamma(s)} L_{i}^{(0)}(s), \qquad 
\phi_{i}(m; s) = \frac{\pi^{s} \abs{m}^{s-\frac12}}{\Gamma(s)} L_{i}^{(m)}(s), \\
\quad\text{where}\quad 
L_{i}^{(m)} = \sum_{c>0} c^{-2s} \sum_{d \bmod{c}}e\left(m\frac{d}{c}\right)
\qquad \left(\bigl(\begin{smallmatrix} * & * \\ c & d\end{smallmatrix}\right)
\in \Gamma_{\xi_i}\bigr).
\end{gathered}
\]
Note that, by the functional equation of $E_i(z,s)$, one has $\phi_{i}(-m; 1-s) = 
\overline{\phi_{i}}(m, s)$.

Now, consider $\left\langle  P_0(\cdot, s, \nP), F\right\rangle$ for $\sigma > 1$. As a function in 
$s$, it has holomorphic continuation to a small neighborhood of the line $\sigma = \frac12$ and 
there, the functional equation holds \citep[cf.][Lemma 5, p.\ 115f]{Good}:
\begin{equation}\label{eq:good_one}
 \begin{aligned}
\left\langle  P_0(\cdot, s, \nP), F\right\rangle & 
 = \frac{1}{2s-1} \sum_{i=1}^\kappa \phi_i (-\nP; 1-s) \left\langle  E_i(\cdot, s), F\right\rangle 
  + \left\langle  P_0(\cdot, s-1, \nP), F\right\rangle \\
 & + \left(  4 \pi \right)^{\frac12  - k} \left( \frac{n}{4} \right)^{s - \frac12}%
\frac{\Gamma(k+s-1)}{\Gamma(s + \frac12)}
\sum_{n\geq 1} \frac{c_n}{( n + \frac{\nP}{2})^{k+s-1} } \Delta_\nP(s, n),
 \end{aligned}
\end{equation}
where 
 $\Delta_{\nP}(s,n)$ is given by \cite[see][p.\ 119]{Good},
 \begin{equation} \label{eq:delta_nsl}
 \begin{aligned}
 \Delta_\nP(s,n) \vcentcolon = 1 & - F\left( \frac{k+s-1}{2}, \frac{k+s}{2}, s+ \frac12; \left( 
\frac{\nP}{2n + \nP}\right)^2 \right)  \\
& \; - \quad\frac{\Gamma( k-s )\Gamma(s - \frac12)}{\Gamma(k+s-1)\Gamma(\frac12 -s)}
\left( \frac{4n + 2\nP}{\nP}\right)^{2s-1} \times  \\ 
& \qquad \times\left[ 
F\left( \frac{k-s}{2}, \frac{k-s+1}{2}, \frac32 - s; \left( \frac{\nP}{2n + \nP}\right)^2 \right)
- 1 \right]. 
 \end{aligned}
 \end{equation}
 Here, $F(a,b,c; z)$ is the hypergeometric function 
 \[
  F(a,b,c; z)  =
 \sum_{w=0}^\infty \frac{(a)_w (b)_w}{ (c)_w w! } z^w \quad\left( \abs{z} <1 \right), \qquad
 (a)_w = \frac{\Gamma(a + w)}{\Gamma(a)}.
 \]
From the expression in \eqref{eq:delta_nsl} one sees immediately that the series on the right hand 
side of \eqref{eq:good_one} is convergent at least for $-1<\sigma<2$. 

Also, note that the Dirichlet series satisfies 
\citep[cf.][Lemma 5 (i), p.\ 116]{Good}
\begin{equation} \label{eq:DfromPandf} 
D(s+k-1,\nP) = \left( 4 \pi \right)^{k-\frac12} \left(\frac{4}{\nP}\right)^{s-{1/2}}
\frac{\Gamma(s + \frac12)}{\Gamma( k + s -1)} 
\left\langle P_0(\cdot, s, \nP), F \right\rangle\qquad (\sigma > 1).
\end{equation}
Since by hypothesis, $D(s,\nP)$ converges for all $s \in \C$, we obtain from 
\eqref{eq:good_one}, \eqref{eq:delta_nsl} and \eqref{eq:DfromPandf} by continuation, that the identity 
 \begin{equation}\label{eq:number4} 
 \begin{aligned}
    \left( 4\pi\right)^{\frac12 - k} & \left( \frac{\nP}{4} \right)^{s-\frac12} 
\frac{\Gamma(k+s-1)}{\Gamma(s + \frac12)} D(k + s - 1, \nP) =  \\ 
 & \frac{1}{2s-1} \sum_{i=1}^\kappa \phi_{i}(-\nP; 1-s) \left\langle E_i(\cdot,s), F \right\rangle 
\\
 & \;  + (4\pi)^{\frac12-k} \left( \frac{\nP}{4}\right)^{\frac12 
- s}\frac{\Gamma(k-s)}{\Gamma(\frac{3}{2}  -s)} \cdot D(k-s,\nP) \\
& \; +  \left( 4\pi\right)^{\frac12 - k} 
\left( \frac{\nP}{4} \right)^{s-\frac12} 
\frac{\Gamma(k+s-1)}{\Gamma(s + \frac12)}
\sum_{n\geq 1} \frac{c_n}{\left( n + \frac{\nP}{2}\right)^{k+s-1}} \cdot \Delta_\nP(s,n)
 \end{aligned}
\end{equation}
holds between meromorphic functions for all $s \in \C$.

Now, take residues at $s= k  + 2m$ $(m \in \N_0)$  on both sides of \eqref{eq:number4}.
By \cite[Theorems 4.3.4, 4.3.5, p.\ 43]{Kub}, the Eisenstein series $E_i(z,s)$ are holomorphic in 
$\sigma > \frac12$ except for finitely many poles in the interval $(\frac12, 1]$, which are 
precisely the poles of the constant coefficients $\phi_{i}(s)$. It follows that the 
$\phi_{i}(\nP;s)$ are holomorphic at $s = k +2m$ and do not contribute to the residue.  

Thus, we pick up non-zero residues only from the second and third terms on the 
right-hand side of \eqref{eq:number4}; in the third term, the only non-zero contribution
comes from the second (non-constant) term of $\Delta_{\nP}(s,n)$ in \eqref{eq:delta_nsl}.
Note that $\operatorname{res}_{s=-m} \Gamma(s) = \frac{(-1)^m}{m!}$ for $m \in \N_0$. 

More precisely, we have
\begin{multline*}
0 = \left( \frac{\nP}{4} \right)^{\frac12-k-2m} \frac{1}{\Gamma(\frac{3}{2}  -k -2m)} 
\frac{1}{(2m)!}
\cdot D(-2m,\nP)  \\
 - \,\left( \frac{\nP}{4} \right)^{k+2m - \frac12} \frac{\Gamma(2k + 2m -1)}{\Gamma(k +2m + 
\frac12)} 
\sum_{n\geq 1} \frac{c_n}{\left( n + \frac{\nP}{2} \right)^{2k+2m-1}} 
\frac{(2 m n)^{-1}\cdot\Gamma(k+2m - \frac12)}{\Gamma(2k+2m-1)\cdot\Gamma(\frac12 - k -2m)} \times\\
\times \left( \frac{4n + 2\nP}{\nP}\right)^{2k + 4m -1} \left[ 
 F\left( -m,  -m + \frac{1}{2}, \frac32 - k -2m, \left( \frac{\nP}{2n + \nP}\right)^2 \right)
- 1 \right]. 
\end{multline*}
We note that 
\begin{align*}
 \Gamma\left({\textstyle\frac{3}{2}} - k -2m\right) & = \left({\textstyle\frac12} - k -2m\right)\cdot 
\Gamma\left({\textstyle\frac12} - k 
-2m\right) \\ \shortintertext{and}
\Gamma\left(k +2m +{\textstyle\frac12}\right) & = 
 \left(k + 2m -{\textstyle\frac12}\right)\cdot 
\Gamma\left( k + 2m - {\textstyle\frac12}\right).
\end{align*}
Thus, we obtain
\[
\begin{aligned}
0   =  & \left( \frac{\nP}{4} \right)^{\frac12-k-2m} \frac{1}{(\frac12 - k -2m)} D(-2m,\nP)  \\
 & - \left( \frac{\nP}{4} \right)^{k+2m - \frac12} \frac{1}{(k  + 2m - \frac12)} 
\sum_{n\geq 1} \frac{c_n}{\left( n + \frac{\nP}{2} \right)^{2k+2m-1}}\;\times \\
  &\qquad \times  \left( \frac{4n + 2\nP}{\nP}\right)^{2k +4m -1}
  \left[ 
 F\left( -m,  -m + \frac{1}{2}, \frac32 - k -2m; \left( \frac{\nP}{2n + 
\nP}\right)^2 \right)
- 1 \right]. 
\end{aligned}
\]
Whence, further 
\[
\begin{aligned}
0    = &  \left( \frac{\nP}{4} \right)^{\frac12-k-2m} D(-2m,\nP)  
  \quad + \; \left( \frac{\nP}{4} \right)^{k + 2m - \frac12} 
  \sum_{n\geq 1} \frac{c_n}{\left( n + \frac{\nP}{2} 
\right)^{2k+2m-1}}\;\times \\
  & \; \times \left( \frac{4n + 2\nP}{\nP}\right)^{2k +2m -1} \left[ 
 F\left( -m,  -m + \frac{1}{2}, \frac32 - k -2m; \left( \frac{\nP}{2n + \nP}\right)^2 \right)
- \;1\; \right]. 
\end{aligned}
\]
Also, we have
\[
 \frac{1}{\left( n + \frac{\nP}{2} \right)^{2k+2m-1}} \cdot \left( \frac{4n + 2\nP}{\nP}\right)^{2k 
+4m  -1}
= \quad \left( \frac{4}{\nP} \right)^{2k+4m-1}  \left( n  + \frac{\nP}{2} \right)^{2m}.
\]
Therefore altogether, we obtain
\[
D(-2m, \nP) =      \sum_{n \geq 1} c_n\left( n + \frac{\nP}{2}\right)^{2m}\!\!\cdot\left[\, -1 + \;
F\left( -m,  -m + \frac{1}{2}, \frac32 - k -2m; \left( \frac{\nP}{2n + \nP}\right)^2 
\right)
\right], 
\]
for all $\nP \in \N_0$. It follows that for all $\nP$
\begin{equation*}
 0 = \sum_{n\geq 1} c_n \left( 2n +  \nP \right)^{2m}  F\left( -m,  -m + \frac{1}{2}, 
\frac32 - k -2m; \left( \frac{\nP}{2n + \nP}\right)^2 \right). 
\end{equation*}

By the definition of the hypergeometric function, we have 
\[
\begin{aligned}
\MoveEqLeft F(-m, -m + {\textstyle{ \frac12} }, {\textstyle{ \frac32}} -k - 2m; z) =  \\
& \sum_{w\geq 0}
\frac{(-m) \dotsm (-m+w-1) ( - m + {\textstyle \frac12})\dotsm( - m +w-1+ {\textstyle 
\frac12})}{ ( -k - 2m + {\textstyle \frac32})\dotsm( - k -2m + {\textstyle \frac32} + w -1) }
\frac{z^w}{w!}.
\end{aligned}
\]
This is a polynomial in $z$ of degree $m$, which we denote as $P_m(z)$. All of its  
coefficients $\lambda_{w,m}$, $(0\leq w\leq  m)$ are non-zero. 
Set $\alpha_{\nu, m} = \nP^{2\nu} \lambda_{\nu,m }$ $(0\leq \nu \leq m)$ and write
$P_m(z) = \sum_{\nu = 0}^m \alpha_{\nu,m}z^{2m-2\nu}$.
We find that 
\[
\sum_{n\geq 1} c_n P_m(2n + \nP) = 0. 
\]
Since all the $\alpha_{\nu,m}$ are non-zero, we find immediately that for all $\nu \in \N_0$, 
\begin{equation}\label{eq:lin_rel_cn} 
 \sum_{n\geq 1} c_n (2n + \nP)^{2\nu} = 0.
\end{equation}
Now recall that (by hypothesis) almost all the coefficients  $c_n$ are $\geq 0$. 
Trivially, if there are non negative coefficients, since the right hand side is zero, there can be no positive coefficients, either and so
$c_n = 0$ for all $n$,  in which case we are finished. 
Thus, we assume that a finite number of coefficients,  $c_{n_1}, \dotsc, c_{n_t}$, with 
$n_1<n_2< \dotsb <n_t$ are negative.
For all $\nu\in \N_0$, we can write the equation \eqref{eq:lin_rel_cn} in the form
\[
\sum_{\substack{ n\geq 1 \\ n\neq n_1, \dotsc, n_t }} c_n (2n + \nP)^{2\nu} =
-c_{n_1}(2n_1 + \nP)^{2\nu} - \dotso  -c_{n_t}(2n_t +\nP)^{2\nu}
\]
Whence for all $\nu$, 
\[
\sum_{\substack{ n\geq 1 \\ n\neq n_1, \dotsc, n_t }} 
c_n \left(\frac{2n + \nP}{2n_t +n} \right)^{2\nu} =
-c_{n_1}\left(\frac{2n_1 + \nP}{2n_t +\nP} \right)^{2\nu} - \dotso  -c_{n_t}. 
 \]
Here, as $\nu$ grows arbitrarily large, the right-hand side is positive and bounded, 
whereas,on the left-hand side, there exist $n$ with $n > n_t$ (and with $c_n>0$), thus, the 
left-hand side becomes arbitrarily large. This is a contradiction, hence $c_n = 0$ for 
all $n>n_t$.

Now, from \eqref{eq:lin_rel_cn} we get the system of linear equations with $0\leq \nu \leq n_t -1$ 
\[
\sum_{n=1}^{n_t} c_n \bigl(2n + \nP\bigr)^{2\nu} = 0.
\]
the determinant of which is easily seen to be non-zero. 
It follows that the remaining coefficients $c_1, \dotsc, c_{n_t}$ vanish, too. Thus, to conclude, $c_n= 0$ for all $n\geq 1$.
\end{proof}

\section{Proof of Theorem \ref{thm:appl_thm}} \label{sec:proofThm2}

Recall that for this theorem, we require $\Gamma$ to be a congruence subgroup of level $N$ 
(then, condition (i) from p.\ \pageref{conditions} is satisfied automatically).
For the proof, we need the following lemma.
\begin{Lemma}\label{lemma:cusp_qeven}
 Let $f$ be cusp form with respect to $\Gamma$, with Fourier coefficients $a(n)$. 
 Let $n_0 \geq 1$. Then, the Fourier expansion
 \[
 g(z) \vcentcolon = \sum_{\substack{ n\geq 1 \\ n \equiv n_0 \bmod{2r}}} a(n) 
 e^{2\pi i n z}
 \]
 defines a cusp form for a congruence subgroup of higher level.
 \end{Lemma}
  \begin{proof}
 With the usual $\mid_k$-operation, we can write $g$ in the form
\[
\begin{aligned}
g(z) & = \frac{1}{\phi(2r)}\sum_{s \bmod 2r} 
e^{- \pi i \frac{n_0 s}{r} z} 
f\mid_k 
\left( \begin{smallmatrix} 
                 1 & \frac{s}{2r} \\ 0 & 1
\end{smallmatrix} \right) \\
 & = \frac{1}{\phi(2r)} \sum_{n \geq n_0} a(n) e^{2\pi i n z} 
 \sum_{ s \bmod  2r}  e^{2\pi i\frac{(n-n_0) s}{2r}} 
 \end{aligned}
\]
Here, $\phi$ denotes the totient function, and summation runs over a minimal  system of 
representatives $\pmod{2r}$. Clearly, the right-hand side is modular with level a multiple of 
$4N$, and is cuspidal if $f$ is.
\end{proof}

Now, we are ready to prove Theorem \ref{thm:appl_thm}. 
\begin{proof}
 Set $c(n) = a(n)a(n+r)$ for all $n\geq 1$. 
 Assume that $c(n) \geq 0$ holds for only finitely many $n$. 
 Then, there is an index $n_0 \in \N$, 
 with the property that $c(n) < 0$ for all $n \geq n_0$. 

Since $c(n_0) = a(n_0)a(n_0 +r) < 0$ it follows that $a(n_0)$ and $a(n_0 + r)$ are both non-zero 
and have opposite sign. 
By induction, it follows that all coefficients $a(n)$ with $n=n_0 + 2rl$   $(l\in\N_0)$
have the same sign. Now, by Lemma \ref{lemma:cusp_qeven}, the resulting sequence of coefficients 
$(a(n))_{n\equiv n_0 (2r)}$ defines a cusp form
\[
g(z) = \sum_{\substack{ n\geq 1 \\ n \equiv n_0 \bmod{2r}} } a(n) e^{2\pi i n z} 
 \]
with real Fourier coefficients, of which, by construction only a finite number are less than or equal to zero. 
However, by the results of Knopp, Kohnen and Pribitkin \cite{KKP03}, Theorem 1, if all Fourier 
coefficients of a cusp form are real, the sequence of its coefficients has infinitely many terms of 
either sign. This is a contradiction.
\end{proof}
\begin{rmk*}\label{rmk:more_general}
Actually, for the result from \cite{KKP03} we use here, the only requirement is that the group 
mentioned in Lemma \ref{lemma:cusp_qeven}, besides being a Fuchsian group of the first kind with 
finite covolume, have $i\infty$ and $0$ as parabolic fixed points. 
However, if $\Gamma$ already satisfies conditions (i)--(iii) (see p.\ \pageref{conditions}), this imposes only a mild further restriction.
Thus, Theorem \ref{thm:appl_thm} holds somewhat more generally then only for 
congruence subgroups.
\end{rmk*}

\bibliographystyle{hplain} 
\bibliography{non-vanishing.bib}

\begin{thebibliography}{1}

\bibitem{Good}
Anton Good.
\newblock On various means involving the {F}ourier coefficients of cusp forms.
\newblock {\em Math. Z.}, 183(1):95--129, 1983.

\bibitem{Hofs}
Jeff Hofstein, Thomas~A. Hulse, and Andre Reznikov.
\newblock Multiple {D}irichlet series and shifted convolutions.
\newblock {\em ArXiv e-prints}, 2015, 1110.4868v3.

\bibitem{Koh99}
Winfried Hohnen.
\newblock Period polynomials and congruences for {H}ecke algebras
  [eigenvalues].
\newblock {\em Proc. Edinburgh Math. Soc. (2)}, 42(2):217--224, 1999.

\bibitem{KKP03}
Marvin Knopp, Winfried Kohnen, and Wladimir Pribitkin.
\newblock On the signs of {F}ourier coefficients of cusp forms.
\newblock {\em Ramanujan J.}, 7(1-3):269--277, 2003.
\newblock Rankin memorial issues.

\bibitem{Kub}
Tomio Kubota.
\newblock {\em Elementary theory of {E}isenstein series}.
\newblock Kodansha Ltd., Tokyo; Halsted Press [John Wiley \& Sons], New
  York-London-Sydney, 1973.

\bibitem{Selb}
Atle Selberg.
\newblock On the estimation of {F}ourier coefficients of modular forms.
\newblock In {\em Proc. {S}ympos. {P}ure {M}ath., {V}ol. {VIII}}, pages 1--15.
  Amer. Math. Soc., Providence, R.I., 1965.

\bibitem{Serre75}
Jean-Pierre Serre.
\newblock Valeurs propres des op\'erateurs de {H}ecke modulo {$l$}.
\newblock In {\em Journ\'ees {A}rithm\'etiques de {B}ordeaux ({C}onf., {U}niv.
  {B}ordeaux, 1974)}, pages 109--117. Ast\'erisque, Nos. 24--25. Soc. Math.
  France, Paris, 1975.

\end{thebibliography}

\noindent \textsc{\small
Mathematisches Institut, Universität Heidelberg, 
Im Neuenheimer Feld 205, \newline {D-69120} Heidelberg, Germany \\}
\textit{E-mail:}  
\href{mailto:hofmann@mathi.uni-heidelberg.de}{\nolinkurl{hofmann@mathi.uni-heidelberg.de}}, 
\href{mailto:winfried@mathi.uni-heidelberg.de}{\nolinkurl{winfried@mathi.uni-heidelberg.de}}

\end{document}